\documentclass{aims}
\usepackage{amsmath}
\usepackage{amscd}
\usepackage{amssymb}
\usepackage{amsthm}
  \usepackage{paralist}
  \usepackage{graphics} 
  \usepackage{epsfig} 
\usepackage{graphicx}  \usepackage{epstopdf}

 \usepackage[colorlinks=true]{hyperref}
\hypersetup{urlcolor=blue, citecolor=red}

\def\currentvolume{X}
 \def\currentissue{X}
  \def\currentyear{200X}
   \def\currentmonth{XX}
    \def\ppages{X--XX}
     \def\DOI{10.3934/xx.xx.xx.xx}

\newcommand{\grad}{\triangledown}

\newcommand{\xo}{x_{0}}
\newcommand{\yo}{y_{0}}
\newcommand{\A}{\alpha}
\newcommand{\B}{\beta}

\newcommand{\laplask}{\mathcal{I}_{\delta}}
\newcommand{\laplasK}{\mathcal{I}^{\delta}}
\newcommand{\laplas}{(-\bigtriangleup)^{1/2}}

\newcommand{\Rn}{\mathbb{R}^n}
 
\newcommand{\SA}{\mathcal{A}}
\newcommand{\SB}{\mathcal{B}}

\newenvironment{Assumptions}
{%
\setcounter{enumi}{0}
\renewcommand{\theenumi}{(\textbf{A}.\arabic{enumi})}
\renewcommand{\labelenumi}{\theenumi}
\begin{enumerate}}%
{\end{enumerate} }

\newcommand{\eps}{\ensuremath{\epsilon}}
\newcommand{\R}{\ensuremath{\mathbb{R}}}

\theoremstyle{plain}
\newtheorem{theorem}{Theorem}[section]
\newtheorem{cor}{Corollary}

\newtheorem{lem}[theorem]{Lemma}
\newtheorem{prop}[theorem]{Proposition}

\theoremstyle{definition}
\newtheorem{defi}[theorem]{Definition}
\newtheorem{rem}{Remark}

\title[Non-local Isaacs equation] 
      {On the differentiability of the solutions of  non-local Isaacs equations involving $\frac 12$-Laplacian.}

\author[Imran H. Biswas and Indranil Chowdhury]{}


 \email{imran@math.tifrbng.res.in}
 \email{indranil@math.tifrbng.res.in}

\thanks{$^*$ Corresponding author: Imran H. Biswas}

\begin{document}
\maketitle


\centerline{\scshape Imran H. Biswas$^*$ and Indranil Chowdhury}
\medskip
{\footnotesize
 \centerline{Centre for Applicable Mathematics, Tata Institute of Fundamental Research}
   \centerline{ P.O.\ Box 6503, GKVK Post Office}
   \centerline{Bangalore 560065, India}
}

\bigskip

 \centerline{(Communicated by the associate editor name)}

\begin{abstract}
We derive $C^{1,\sigma}$-estimate for the solutions of a class of non-local elliptic Bellman-Isaacs equations. These equations are fully nonlinear and are associated with infinite horizon stochastic differential game problems involving jump-diffusions.  The non-locality  is represented by the presence of fractional order diffusion term  and we deal with the particular case of $\frac 12$-Laplacian, where the order $\frac 12$ is known as the critical order in this context.  More importantly, these equations are not translation invariant  and we prove that the viscosity solutions of such equations are $C^{1,\sigma}$, making the equations classically solvable.
\end{abstract}
\section{Introduction}
In this article we investigate regularity properties of a class of fully nonlinear nonlocal elliptic equations of Isaacs type. This class has its origin in the stochastic differential game problems where the state processes are governed by controlled drift-jump processes. The solutions to such equations are interpreted in the viscosity sense and are a priori known to be only continuous. The question of differentiability is a subtle one and we want to establish that the viscosity solutions actually become differentiable and  the derivatives are regular enough to make the equations classically solvable. The nonlocal equations that we are interested in have the form
 
 \begin{align}
   \label{eq:HJB-master} F\big (x, u(x),\nabla u(x), (-\Delta)^{\frac 12} u(x)\big) = 0, \quad x\in \R^n.
 \end{align}  The nonlinearity $F$ is a function from $\R^n\times \R\times \R^n\times \R$ to $\R$ and defined as follows:
 
 \begin{align}
    \label{eq:nonlinearity} F\big (x, r, p, t\big)= \sup_{\alpha\in \mathcal{A}} \inf_{\beta\in\mathcal{B}} \Big\{f^{\alpha, \beta}(x) + c^{\alpha,\beta}(x)r+b^{\alpha,\beta}(x)\cdot p +a^{\alpha,\beta}(x) t\Big\},
 \end{align} where $\mathcal{A}$ and $\mathcal{B}$ are two metric spaces and the functions $f^{\alpha, \beta}(x), c^{\alpha,\beta}(x)$ and $  a^{\alpha,\beta}(x)$ are real valued and $ b^{\alpha,\beta}(x)$ are $\R^n$-valued. The precise set of assumptions on these functions will be listed in the next section, but roughly speaking we will be working in a framework where Eq.\eqref{eq:HJB-master} is well-posed in the  viscosity sense and has Lipschitz continuous solution.
 
   Note that for any $s\in (0, 1)$, the operator $(-\Delta)^s$ has the form (see \cite{Landkof:1966} for example)
   \begin{align}\label{fraclap}
 (-\Delta )^s  u(x) = c(n, s)\int_{\R^n} \frac{u(x)-u(y)}{|x-y|^{n+2s}} \ dy.
\end{align} The constant $c(n,s)$ depends only on $n$ and $s$, and the integral has to be understood in the principal value sense. Besides non-linearity, clearly the representation \eqref{fraclap} makes the problem \eqref{eq:HJB-master} non-local. The presence of such non-locality goes beyond the pure academic curiosity and is motivated by concrete applications.  As mentioned, the application area that we have in mind is stochastic differential games where the state process is governed by  a pure jump type noise and two players are involved in minimizing or maximizing a pay-off over an infinite time horizon. The value function of such a game problem satisfies a dynamic programming principle, which is then used to show that the value function is  the viscosity solution of the underlying Bellman-Isaacs equation. The presence of jumps in the state dynamics results in non-locality of the underlying Bellman-Isaacs equations, and they assume the the form  \eqref{eq:HJB-master}-\eqref{eq:nonlinearity}. We refer to \cite{Biswas2012} for more on stochastic differential games involving jump-diffusions. 

 The question of solvability in the classical sense for such equations is a subtle one and may not always have an affirmative answer. This prompts one to look for an appropriate concept of weak solutions for such problems and the notion of viscosity solutions makes perfect sense here. The viscosity solutions are by definition only continuous and, depending on the nature of the problem, it is often fairly simple to establish H\"{o}lder or Lipschitz regularity for such solutions. However, in order for classical solvability, one needs to establish differentiability of viscosity solutions and, in addition, prove that the derivative is H\"{o}lder continuous. 
 
      Since the pioneering development of viscosity solution theory for fully non-linear PDEs, the same has been extended to the equations that have non-localities in them.
The articles \cite{BCI:P07, BI:P07, Dro2006,Imbert: 2005ft, Sayah:2010gt, Jakobsen:2005jy, silvestre:2010, silvestre:2012} and the references therein provide a rich source for results on the viscosity solution theory for non-local equations. The issues addressed in these articles range from standard existence,  uniqueness and stability theory to certain types of  regularity questions. The question of differentiability or $C^{1,\sigma}$-type regularity for viscosity solutions of nonlocal equations is a relatively newer area of pursuit and most of the developments have taken place in the last decade alone. In regard to this frontier, the work of Caffarelli et al. have been milestones of sorts and we would like to particularly mention \cite{cafarelli:2009, cafarelli:2010, cafarelli:2011, cafarelli:2011.2} and \cite{davil2012} for developments in this area.  These studies are largely restricted to the Dirichlet problem
   \begin{align}
\label{eq:sub-critical case} \sup_{\alpha\in \mathcal{A}} \inf_{\beta\in\mathcal{B}} \Big\{f^{\alpha, \beta}(x) + \int_{\R_y^n} \big(u(x)-u(x+y)\big) K^{\alpha, \beta}(x,y)\,dy\Big\} & = 0~\text{in}~\Omega,\\
        \notag u &= g \quad \text{in}\quad \R^n\backslash \Omega,
 \end{align} where $\Omega$ is an open subset in $\R^n$ and $K^{\alpha, \beta}(x,y)$ is a suitable kernel satisfying the condition  $\frac{\lambda(1-s)}{|y|^{n+2s}}\le K^{\alpha, \beta}(x,y)  \le \frac{\Lambda (1-s)}{|y|^{n+2s}}$ for all $x$. Here  $s$ is a given constant from $(0,1)$ and  $\lambda, \Lambda$ are given positive constants. At an intuitive level, \eqref{eq:sub-critical case}  translates roughly into an equation of the form 
  \begin{align}
   \label{eq:HJB-secondary} F\big (x, (-\Delta)^{s} u(x)\big) = 0, \quad x\in \Omega.
 \end{align} In the case where $F$ is translation invariant, i.e  $F$ is independent of $x$, the regularity issue is comprehensively settled by Caffarelli $\&$ Silvestre \cite{cafarelli:2009}. The problem is more subtle for the non-translation invariant case and Krylov-Safanov type  regularity results are made available in \cite{cafarelli:2009} only if $s > \frac 12$. This, however, does not address the issue if the Hamiltonian has additional dependence on the gradient as well. We are interested in Lipschitz continuous viscosity solutions and, for $s < \frac 12$, such solutions would qualify as classical solutions of \eqref{eq:HJB-secondary}. The case of $s =\frac 12 $ is the borderline case. The Lipschitz continuity of viscosity solutions is not enough to make it a classical solution in this case and it is rightly termed as the critical case. Moreover,  the gradient term in \eqref{eq:HJB-master} balances out the $\frac 12$-Laplacian term at every scale and it is a priori not clear if the non-local term will have additional regularizing effect.  However, in the context of parabolic (time-dependent) HJB equations with fractional Laplacian, one has classical parabolic regularization if $s > \frac 12$
  (see \cite{Imbert: 2005ft}), $C^{1,\sigma}$ regularization if $s=\frac 12$  (see \cite{Biswas2013, silvestre:2010, silvestre:2012} ),  and no regularization if $s< \frac 12$ (see \cite{Kiselev:2008} ). The results in  \cite{ silvestre:2010} are available only for the translation invariant problem and some relevant (but marginal) extensions are available in \cite{Biswas2013, silvestre:2012}. The case of a general non-translation invariant problem is largely open.

   It is fairly simple to see that, if \eqref{eq:HJB-master} is translation invariant, the viscosity solution of  \eqref{eq:HJB-master} is actually a constant and therefore the question of regularity is redundant. The interesting case therefore is when the problem   \eqref{eq:HJB-master} is not translation invariant. In this case, under appropriate conditions, it is also fairly simple to establish Lipschitz continuity for viscosity solutions of \eqref{eq:HJB-master}, and the question of differentiability for such solutions comes as a natural one. In this article, we draw inspirations from the recent regularity results from \cite{silvestre:2012, silvestre:2010} and show that the presence of critical order fractional diffusion in \eqref{eq:HJB-master} makes the unique Lipschitz continuous viscosity solution differentiable and the derivative is H\"{o}lder continuous. In other words, the problem becomes classically solvable.
   
   The rest of this paper is organized as follows. We state the assumptions, 
detail the technical framework and state the main results in Section \ref{technical}.    
Section \ref{oscillation} constitutes the main technical part and we establish a diminishing of oscillation lemma here .  The Section \ref{holder-regularity} is the final one and $C^{0,\sigma}$-regularity for the derivative is established here. The paper is concluded with an appendix where we prove the existence and Lipschitz continuity of viscosity solution of \eqref{eq:HJB-master}.

 \section{Preliminaries, framework and main results}\label{technical}
 We begin with the description of notation that will be frequently used throughout the paper. The Euclidean norm in any $\R^d$ type space is denoted by $|\cdot|$. We use the letters $C, K, N$ etc to denote various generic constants depending on the data. We use the notation $B_r(x)$ to denote an open  ball in $\R^n$ of 
radius $r$ around  $x$; and we write $B_r$ for $B_r(0)$. Also, we use the notation $Q_r$ to denote the space time cylinder $[-r, 0]\times B_r$. For any subset $U\subset \R^n $, the space of all bounded and H\"{o}lder continuous function with exponent $\sigma$ is denoted by $C^{0,\sigma}(U)$ and is equipped with the norm 
  $$ \qquad |u|_{C^{0,\sigma}(U)} = |u|_{0} 
+ \displaystyle\sup_{x,y\in U} \dfrac{|u(x)-u(y)|}{|x-y|^{\sigma}}~\text{where}~|u|_0 = \displaystyle\sup_{x\in U} |u(x)|.$$  The space of bounded functions with H\"{o}lder continuous derivatives of exponent $\sigma$ is denoted by $C^{1,\sigma}(U)$ and equipped with norm
$$ ||u||_{C^{1,\sigma}(U)} = |u|_{0}+ |\nabla u|_{C^{0,\sigma}}. $$ We use $USC(U)$, $LSC(U)$ and $C(U)$ to denote the space of upper semicontinuous, lower semicontinuous and continuous
functions on $U$ respectively. A lower index $p$ denotes the polynomial growth at infinity, so $C_p(U) ,USC_p(U),
 LSC_p(U) $ consist of functions $u$ from $C(U) ,USC(U), LSC(U)$ satisfying the growth condition 
$$ |u(x)| \leq C(1+ |x|^{p}) \quad \mbox{for all} \quad x\in U. $$
Also denote the bounded continuous, upper semicontinuous and lower semicontinuous functions by $C_b(U), USC_b(U)$ and 
 $ LSC_b(U) $, respectively.  We now list the assumptions under which \eqref{eq:HJB-master} is wellposed in the viscosity sense and has Lipschitz continuous viscosity solutions.  

\begin{Assumptions}
\item\label{A1} The spaces $\mathcal{A}$ and $\mathcal{B}$ are two compact metric spaces; the functions $a^{\alpha, \beta}(x)$, $ c^{\alpha, \beta}(x), f^{\alpha, \beta}(x)$ are real valued and $b^{\alpha, \beta}(x)$ are $\R^n$-valued continuous functions with respect to the variables $\alpha, \beta$ and $x$.
\vspace{.2cm}
\item\label{A2} The functions $a^{\alpha, \beta}(x)$, $b^{\alpha, \beta}(x), c^{\alpha, \beta}(x), f^{\alpha, \beta}(x)$ are bounded and Lipschitz continuous in $x$ and there is a constant $K$ such that 
   $$\Big(|a^{\alpha, \beta}(\cdot)|_{C^{0,1}(\R^n)}+|b^{\alpha, \beta}(\cdot)|_{C^{0,1}(\R^n)}+|c^{\alpha, \beta}(\cdot)|_{C^{0,1}(\R^n)}+|f^{\alpha, \beta}(\cdot)|_{C^{0,1}(\R^n)}\Big) \le K$$
   for all $(\alpha, \beta)\in \mathcal{A}\times \mathcal{B}$.
   \vspace{.2cm}
   \item\label{A3} There is a positive constant $\lambda$ such that $ c^{\alpha, \beta}(x) > \lambda$ for all $(\alpha, \beta)$ and $x$.
   \vspace{.2cm}
     \item\label{A4} There is a positive constant $\lambda_1$ such that $ a^{\alpha, \beta}(x) > \lambda_1$ for all $(\alpha, \beta)$ and $x$.
     
   \end{Assumptions}
   
   \begin{rem}
 The assumptions \ref{A1}-\ref{A4} are natural and standard in view of the optimal control/game problems for jump-diffusions. The condition \ref{A3} is essential to prove comparison principle for \eqref{eq:HJB-master} and would ensure well-posedness in the viscosity sense. In a general scheme of work \ref{A3} ensures some H\"{o}lder type regularity for viscosity solutions. However, we need a priori Lipschitz continuity  and this is ensured by assuming that the constant $\lambda$ in \ref{A3} is sufficiently large.
 \end{rem}
 
 \subsection{Viscosity solution framework.} There are many ways to formulate the notion of viscosity solutions for nonlocal problems, all them however lead to the same solution under standard assumptions. We follow the formulation from \cite{Jakobsen:2005jy} to define the viscosity solutions of \eqref{eq:HJB-master}, and this involves the following quantities.  For $\kappa\in (0,1)$, let
     \begin{align*}
      \mathcal{I}_\kappa(\varphi)(x) &= -c(n, 1/2) \int_{B(0,\kappa)}\frac{\big(\varphi(x+z)-\varphi(x)\big)}{|z|^{n+1}} dz,\\
        \mathcal{I}^{\kappa}(u)(x) &= -c(n,1/2) \int_{B(0,\kappa)^{c}}\frac{\big(u(x+z)-u(x)\big)}{|z|^{n+1}} dz,
     \end{align*} and by the representation \eqref{fraclap}, it holds that 
     
       \begin{align*}
(-\Delta)^{\frac{1}{2}}\varphi =    \mathcal{I}_\kappa(\varphi)+   \mathcal{I}^{\kappa}(\varphi)
    \end{align*} for every  $\kappa \in (0,1)$. We now define the notion of viscosity solution for \eqref{eq:HJB-master}.

         \begin{defi}[viscosity solution]\label{defi:viscosolution}
      \begin{itemize}
        \item[$i.)$] A function $u\in USC_b(\mathbb{R}^n)$ is a viscosity subsolution of  \eqref{eq:HJB-master} if for any $\varphi\in C^{2}(\mathbb{R}^n)$, whenever $x \in \mathbb{R}^n$ is a global maximum point of $u-\varphi$ it holds that
        \begin{align*}
     &  \sup_{\alpha\in \mathcal{A}} \inf_{\beta\in\mathcal{B}} \Big\{f^{\alpha, \beta}(x) + c^{\alpha,\beta}(x)u(x)+b^{\alpha,\beta}(x)\cdot\nabla \varphi(x) \\&\hspace{5cm}+a^{\alpha,\beta}(x)   \mathcal{I}_\kappa(\varphi)(x)+a^{\alpha,\beta}(x)   \mathcal{I}^\kappa(u)(x)\Big\} \le 0
        \end{align*}  for all $\kappa\in (0,1)$.
          \item[$ii.)$] A function $u\in LSC_b(\mathbb{R}^n)$ is a viscosity supersolution of  \eqref{eq:HJB-master} if for any $\varphi\in C^{2}( \mathbb{R}^n)$, whenever $x \in\mathbb{R}^n$ is a global minimum point of $u-\varphi$, it holds that
        \begin{align*}
        & \sup_{\alpha\in \mathcal{A}} \inf_{\beta\in\mathcal{B}} \Big\{f^{\alpha, \beta}(x) + c^{\alpha,\beta}(x)u(x)+b^{\alpha,\beta}(x)\cdot\nabla \varphi(x) \\&\hspace{5cm}+a^{\alpha,\beta}(x)   \mathcal{I}_\kappa(\varphi)(x)+a^{\alpha,\beta}(x)   \mathcal{I}^\kappa(u)(x)\Big\}  \ge 0
        \end{align*}
        for all $\kappa\in (0,1)$.
           \item[$iii.)$] A function $u\in C_b( \mathbb{R}^n)$ is a viscosity solution of \eqref{eq:HJB-master} if it is both a sub and supersolution.
     \end{itemize}
    \end{defi}
\begin{rem}
    In the above definition, it is easy to see that one can replace ``{\em global minimum/global maximum}" by  ``{\em strict global minimum/global maximum}" , and the resulting definition will still be equivalent to Definition \ref{defi:viscosolution}. Moreover, at the points of maximum or minimum of $u-\varphi$, it would be harmless to assume $u = \varphi$.  
\end{rem}
\begin{rem}
It is to be noted here that the condition $\kappa\in (0,1)$ could be replaced by $\kappa \in (0,\gamma)$ for any positive constant $\gamma$ in the Definition \ref{defi:viscosolution}. For our methodology to work, we need  subsolution/supersolution inequalities for a sequence of $\kappa$'s converging to $0$. 
\end{rem}

 We find it necessary to include the following alternative (but equivalent) definition of viscosity solution which is better suited for existence theory via Perron's method.
 
   \begin{defi}[alternative definition]\label{defi:viscosolution-alt}
     A function $u\in USC_b(\mathbb{R}^n)$ ( $u\in LSC_b(\mathbb{R}^n)$) is a viscosity subsolution ({\em supersolution}) of  \eqref{eq:HJB-master} if for any $\varphi\in C^{2}_b(\mathbb{R}^n)$, whenever $x \in \mathbb{R}^n$ is a global maximum ({\em minimum}) point of $u-\varphi$ it holds that
        \begin{align*}
       &\sup_{\alpha\in \mathcal{A}} \inf_{\beta\in\mathcal{B}} \Big\{f^{\alpha, \beta}(x) + c^{\alpha,\beta}(x)u(x)+b^{\alpha,\beta}(x)\cdot\nabla \varphi(x) \\&\hspace{7cm}+a^{\alpha,\beta}(x)  (-\Delta)^{\frac 12}(\varphi)(x)\Big\} \le 0~ (\ge 0).
        \end{align*}
        
            \end{defi}
            
            It now follows by the method of doubling of variables, standard in the context of viscosity solution theory, that subsolutions are always dominated by supersolutions and the following comparison principle holds.  
            
            \begin{lem}[Comparison principle]
            \label{comparison}
 Let  \ref{A1} - \ref{A4} be true, and $u\in USC_b(\R^n)$ be a viscosity subsolution and 
$v\in LSC_b(\R^n)$ be a viscosity supersolution of \eqref{eq:HJB-master}; then 
$$ u\leq v \quad \mbox{in} \quad \R^n.$$
                
            \end{lem}

\begin{proof}
The proof uses the so-called doubling of variables, and the essential ingredients of this method are detailed in the proof of  Theorem \ref{thm:lipschitz}. This type of comparison principle is available in much more generality and  could be found in \cite{BI:P07}. 
\end{proof}

\begin{theorem}[Existence] \label{thm:exisence}Let \ref{A1}-\ref{A4} be true. There exists unique viscosity solution $u \in 
C_b(\R^n)$ of the equation \eqref{eq:HJB-master}.
\end{theorem}
\begin{proof}
 The uniqueness follows from comparison principle. The proof of existence is standard and uses the classical Perron's method. However, for the sake completeness of our presentation, we add a detailed proof in Appendix \ref{appendix}.
\end{proof}

  It is fairly straightforward to prove that, under \ref{A1}-\ref{A4}, the unique viscosity solution is H\"{o}lder continuous of some exponent $\gamma$. However, we are able to deal with solutions that are Lipschitz continuous and this could be ensured by assuming that the constant $\lambda$ in \ref{A3} is large enough.

\begin{theorem}[Lipschitz continuity]\label{thm:lipschitz}
Assume that \ref{A1}-\ref{A4} hold and $u\in C_b(\R^n)$ be the unique viscosity solution of $$\displaystyle \sup_{\A \in \SA} \inf_{\beta \in \SB} \{ f^{\A , \beta} (x) + c^{\A , \beta} (x) u(x) + 
b^{\A , \beta} (x)\cdot\grad u(x) + a^{\A , \beta} (x) \laplas u(x) \} = 0.$$
Then there is a constant $\lambda_0$ depending on the constants in \ref{A1}, \ref{A2} and \ref{A4} such that   if $\lambda > \lambda_0$ in \ref{A3}, then 

$$ |u(x) - u(y)|\leq L |x-y| \quad \mbox{for each }  \ x,y \in \R^n $$
where $L$ is a positive constant depending on  \ref{A1}-\ref{A4}.

\end{theorem}
\begin{proof} The proof is deferred to Appendix \ref{appendix}.
\end{proof}

\subsection{The main result and outline of our strategy} The main contribution in this article is stated as follows. 
\begin{theorem}[$C^{1,\sigma}$-regularity] \label{main-theorem}Let \ref{A1}-\ref{A4} hold and the constant $\lambda$ in \ref{A3} be large enough such that the unique viscosity solution $u$ of \eqref{eq:HJB-master} is Lipschitz continuous. Then $u$ is continuously differentiable and there is a constant $\sigma\in (0,1)$ and a constant $C$ such that 
$$| \nabla u(x)-\nabla u(y) | \le C|x-y|^{\sigma}$$

for all $x, y\in \Rn$.
\end{theorem} The detailed proof of the main theorem is given in Section \ref{holder-regularity}. In the rest of this section, we will briefly outline our strategy.  

 Our aim here is to stay close and follow the methodology used in \cite{silvestre:2010, silvestre:2012} to prove $C^{1,\sigma}$-type regularity estimates for time dependent HJB-equations with $\frac 12$-Laplacian. Besides the absence of time variable, the most crucial difference that we have with \cite{silvestre:2010} is that our problem is not translation invariant. The methodology in \cite{Biswas2013, silvestre:2012} allows Hamiltonians that are not translation invariant in a limited sense. In comparison, \eqref{eq:HJB-master} seems to be loaded with lack of translation invariance in all its constituent components.  However, the special structure of the fractional order term in \eqref{eq:HJB-master} allows us to reduce this `degree' of non-translation invariance and we have following lemma to this end. 
 
\begin{lem}\label{lem:reduction-1}
         A function $u\in USC_b(\Rn)$ ($LSC_b(\Rn)$) is a subsolution ({\em supersolution}) of \eqref{eq:HJB-master} iff it satisfies 
          \begin{align*}
       \sup_{\alpha\in \mathcal{A}} \inf_{\beta\in\mathcal{B}} \Big\{\tilde{f}^{\alpha, \beta}(x) + \tilde{c}^{\alpha,\beta}(x)u(x)+\tilde{b}^{\alpha,\beta}(x)\cdot\nabla u(x) +(-\Delta)^{\frac 12} u(x)\Big\} \le 0~ (\ge 0)
        \end{align*} in the viscosity sense, where $\Big({\tilde{f}}^{\alpha, \beta}, {\tilde{c}}^{\alpha, \beta}, {\tilde{b}}^{\alpha, \beta}\Big) = \frac{1}{a^{\alpha, \beta} } \Big({f}^{\alpha, \beta}, {c}^{\alpha, \beta}, {b}^{\alpha, \beta}\Big). $ 
\end{lem}

\begin{proof} The proof is elementary and only requires basic reasoning. We will only argue for the subsolution part.  Let $\varphi\in C_b^2(\Rn)$ be such that $u-\varphi$ has a global maximum at $x$. Then 
    \begin{align*}
       &\sup_{\alpha\in \mathcal{A}} \inf_{\beta\in\mathcal{B}} \Big\{f^{\alpha, \beta}(x) + c^{\alpha,\beta}(x)u(x)+b^{\alpha,\beta}(x)\cdot\nabla \varphi(x) +a^{\alpha,\beta}(x)  (-\Delta)^{\frac 12}(\varphi)(x)\Big\} \le 0\\
     \Leftrightarrow   &  \inf_{\beta\in\mathcal{B}} \Big\{f^{\alpha, \beta}(x) + c^{\alpha,\beta}(x)u(x)+b^{\alpha,\beta}(x)\cdot\nabla \varphi(x) +a^{\alpha,\beta}(x)  (-\Delta)^{\frac 12}(\varphi)(x)\Big\} \le 0
        \end{align*} for all $\alpha \in\mathcal{A}$. Now fix $\alpha\in \mathcal{A}$. Then for every $\eps > 0$, there is $\B=\B(\eps, \alpha)$ such that
        \begin{align*}
      & f^{\alpha, \beta}(x) + c^{\alpha,\beta}(x)u(x)+b^{\alpha,\beta}(x)\cdot\nabla \varphi(x) +a^{\alpha,\beta}(x)  (-\Delta)^{\frac 12}(\varphi)(x)\le \epsilon\\
     \Rightarrow ~~& {\tilde{f}}^{\alpha, \beta}(x) + {\tilde{c}}^{\alpha,\beta}(x)u(x)+{\tilde{b}}^{\alpha,\beta}(x)\cdot\nabla \varphi(x) +(-\Delta)^{\frac 12}(\varphi)(x)\le \frac{\epsilon}{\lambda_1}\\
       \Rightarrow ~~ &\inf_{\B\in\mathcal{B}} \Big\{ {\tilde{f}}^{\alpha, \beta}(x) + {\tilde{c}}^{\alpha,\beta}(x)u(x)+{\tilde{b}}^{\alpha,\beta}(x)\cdot\nabla \varphi(x) \Big\}+(-\Delta)^{\frac 12}(\varphi)(x)\le \frac{\epsilon}{\lambda_1} 
        \end{align*} for all $\alpha\in \mathcal{A}$. Therefore
        
        \begin{align*}
         \sup_{\alpha\in\mathcal{A}}\inf_{\B\in\mathcal{B}} \Big\{ {\tilde{f}}^{\alpha, \beta}(x) + {\tilde{c}}^{\alpha,\beta}(x)u(x)+{\tilde{b}}^{\alpha,\beta}(x)\cdot\nabla \varphi(x) \Big\}+(-\Delta)^{\frac 12}(\varphi)(x)\le \frac{\epsilon}{\lambda_1} 
        \end{align*} for $\epsilon > 0$. Therefore, we conclude that 
         \begin{align*}
         \sup_{\alpha\in\mathcal{A}}\inf_{\B\in\mathcal{B}} \Big\{ {\tilde{f}}^{\alpha, \beta}(x) + {\tilde{c}}^{\alpha,\beta}(x)u(x)+{\tilde{b}}^{\alpha,\beta}(x)\cdot\nabla \varphi(x) \Big\}+(-\Delta)^{\frac 12}(\varphi)(x)\le0       
          \end{align*} This completes the proof once we observe that the steps above are reversible.  
        
\end{proof}

 The above lemma clearly shows that solving \eqref{eq:HJB-master} in the viscosity sense is equivalent to solving 
  \begin{align}
        \label{eq:reduction}  \sup_{\alpha\in \mathcal{A}} \inf_{\beta\in\mathcal{B}} \Big\{\tilde{f}^{\alpha, \beta}(x) + \tilde{c}^{\alpha,\beta}(x)u(x)+\tilde{b}^{\alpha,\beta}(x)\cdot\nabla u(x) \Big\}+(-\Delta)^{\frac 12} u(x) =0.
 \end{align} It is also obvious, once we invoke the a priori Lipschitz continuity of unique viscosity solution of \eqref{eq:reduction}, the function $g^{\alpha, \beta}(x) = \big({\tilde{c}}^{\alpha, \beta}(x) -\lambda\big) u(x)+ \tilde{f}^{\alpha, \beta}(x)$ satisfies the same assumptions as $f^{\alpha, \beta}$ and the equation 
   \begin{align}
        \label{eq:reduction-2}  \sup_{\alpha\in \mathcal{A}} \inf_{\beta\in\mathcal{B}} \Big\{g^{\alpha, \beta}(x) +\tilde{b}^{\alpha,\beta}(x)\cdot\nabla u(x) \Big\}+\lambda u(x)+(-\Delta)^{\frac 12} u(x) =0.
 \end{align} has unique viscosity solution which coincides with the unique solution of \eqref{eq:reduction}, which is same the unique solution of \eqref{eq:HJB-master}. Clearly, a $C^{1,\sigma}$ type regularity estimates for Lipschitz continuous solutions of equations of the form \eqref{eq:reduction-2} is sufficient for establishing Theorem \ref{main-theorem}. The equation \eqref{eq:reduction-2} has the form
    \begin{align}
        \label{eq:reduction-3}  H(x,\nabla u(x))+\lambda u(x)+(-\Delta)^{\frac 12} u(x) =0,
        \end{align} where the Hamiltonian $H(x, p)$ is Lipschitz continuous in both its variables,  and we have the a priori knowledge that \eqref{eq:reduction-3} has Lipschitz continuous solution.  The form \eqref{eq:reduction-3} permits us to apply techniques used for parabolic problems \cite{silvestre:2012}, and we motivate our approach as follows. 
        
         For the moment assume that \eqref{eq:reduction-3} has smooth solutions, and $H$ is also smooth. Then for any $\ell \in \Rn $, we differentiate the equation \eqref{eq:reduction-3} with respect to $x$ variable along $\ell$ to obtain
\begin{align}
\label{eq:linerized} D_p H(x, \grad u)\cdot \grad(\partial_\ell u) + D_x H(x,\grad u)\cdot \ell + \lambda \partial_\ell u + \laplas (\partial_\ell u) = 0; 
\end{align}
where $\partial_\ell u$ is the directional derivative of $u$ along $\ell$-direction. We rewrite equation \eqref{eq:linerized} as 
\begin{align} \label{aux}
A(x)\cdot\grad v + \lambda v + F(x) + \laplas v = 0,  
\end{align}
where $v= \partial_\ell u$, $A(x)= D_p H(x, \grad u)$ and $F(x)= D_x H(x,\grad u)\cdot \ell$.  At this point, the functions $A(x)$ and $F(x)$ are also unknown. The only a priori information we have here is that they are bounded, thanks to the Lipschitz continuity of $u$ and $H$. More specifically,  
if $L$ is the Lipschitz constant of $u$, then we could find two positive constants  $A$ and $B$ such that $\sup_{x\in \Rn, |p|< L} |D_p H(x, p)| \leq A$ and $\sup_{x\in \R^n, |p|< L}|D_x H(x, p)|\leq B $ and then \eqref{aux} results in the inequalities 

\begin{align}\label{eq:aux1}
-A|\grad v| -B +\lambda v + \laplas v \leq 0 \\
\label{eq:aux2}
A|\grad v| +B +\lambda v + \laplas v \geq 0 .
\end{align}
        Thus if $u$ is a smooth solution of \eqref{eq:reduction-3} it follows that $\partial_\ell u$ would satisfy 
\eqref{eq:aux1} and \eqref{eq:aux2}. A priori the viscosity solution is only Lipschitz continuous  and the inequalities \eqref{eq:aux1} and \eqref{eq:aux2} for the directional derivatives are ill-defined.
 However, at a formal level, \eqref{eq:aux1} and \eqref{eq:aux2} are also the inequalities satisfied by the  difference quotients of smooth solutions. Encouragingly, the difference quotients of a Lipschitz continuous viscosity solution are bounded continuous functions and we establish that  they obey the inequalities  \eqref{eq:aux1} -\eqref{eq:aux2} in the viscosity sense. This allows us to obtain uniform $C^{0,\sigma}$-type estimate for the difference quotients, which is then suitably translated into a regularity estimate for the derivative. The following lemma is a key step in establishing the inequalities \eqref{eq:aux1} and \eqref{eq:aux2}
 within the viscosity solution framework. 
 
 \begin{lem} \label{dif:quo:lem}
Assume \ref{A1}-\ref{A4} hold and $\ell \in \Rn$ be a given vector.  Furthermore, let $u, v$ be two bounded and Lipschitz continuous functions satisfying the inequalities 
\begin{align*}
     &  \sup_{\alpha\in \mathcal{A}} \inf_{\beta\in\mathcal{B}} \Big\{f^{\alpha, \beta}(x+\ell) + c^{\alpha,\beta}(x+\ell)u(x)+b^{\alpha,\beta}(x+\ell)\cdot\nabla u(x) \\& \hspace{7cm}+a^{\alpha,\beta}(x+\ell)  (-\Delta)^{\frac 12}u(x)\Big\} \le 0,\\
&  \sup_{\alpha\in \mathcal{A}} \inf_{\beta\in\mathcal{B}} \Big\{f^{\alpha, \beta}(x) + c^{\alpha,\beta}(x)v(x)+b^{\alpha,\beta}(x)\cdot\nabla v(x) +a^{\alpha,\beta}(x)  (-\Delta)^{\frac 12}v(x)\Big\} \ge 0
\end{align*}
in the viscosity sense. Then there are three positive constants $A, B_0$ and $B_1$ (independent of $\ell$) such that $(u-v)$ satisfies 
\begin{align}
\label{eq:diff-ineq-1}-B_0\sup_{x\in \R^n} |u(x+\ell)-v(x)|-A|\grad(u-v)| -B_1|\ell| + \lambda(u-v) + \laplas (u-v) \leq 0
\end{align}
in viscosity sense. Also if $u, v$  respectively satisfy the inequalities
\begin{align*}
        &\sup_{\alpha\in \mathcal{A}} \inf_{\beta\in\mathcal{B}} \Big\{f^{\alpha, \beta}(x+\ell) + c^{\alpha,\beta}(x+\ell)u(x)+b^{\alpha,\beta}(x+\ell)\cdot\nabla u(x) \\&\hspace{7cm}+a^{\alpha,\beta}(x+\ell)  (-\Delta)^{\frac 12}u(x)\Big\} \ge 0,\\
  &\sup_{\alpha\in \mathcal{A}} \inf_{\beta\in\mathcal{B}} \Big\{f^{\alpha, \beta}(x) + c^{\alpha,\beta}(x)v(x)+b^{\alpha,\beta}(x)\cdot\nabla v(x) +a^{\alpha,\beta}(x)  (-\Delta)^{\frac 12}v(x)\Big\} \le 0
\end{align*}
in viscosity sense, then $(u-v)$ satisfies 
\begin{align}
\label{eq:diff-ineq-2}B_0\sup_{x\in \R^n} |u(x+\ell)-v(x)|+A|\grad(u-v)| +B_1|\ell| + \lambda(u-v) + \laplas (u-v) \geq 0
\end{align}
in the viscosity sense. 
\end{lem}
  A proof of this lemma requires the notion of $\Gamma-$convergence for sup/inf convolution of upper/lower semicontinuous function. We will define this notion and state a relevant convergence result for semiconvex approximation to upper semicontinuous functions. This type of results are fairly standard and are widely used in the viscosity solution theory, a rigorous proof could be found in \cite{LIC}.
  \begin{defi}[$\Gamma$-convergence]
   A sequence $\{u_k\}_k \in LSC(\Rn)$ is said to \textit{$\Gamma$-converge} to $u$ on $\Rn$ if 
    \\
$i.)$ for any sequence ${x_p} \in \Rn$ with $x_p \rightarrow x$, it holds that $\lim \inf_{p\rightarrow\infty} u_p(x_p) \geq u(x) $.\\
$ii.)$ For every $x\in \Rn$ there exists a sequence $x_p$ with $x_p \rightarrow x$ such that\\ $\lim \sup_{p\rightarrow\infty} 
u_p(x_p)= u(x)$.

  \end{defi}

 \begin{prop} \label{prop1}
 Let $u$ be a bounded upper semicontinuous function in $\Rn$ and for  $\varepsilon >0$, 
the sup-convolution of $u$ is defined as 
$u^{\varepsilon} (x) =  \displaystyle \sup_{ y\in \Rn} \left[ u(x+y) - \frac{|y|^2}{\varepsilon} \right] $.  
The sequence $\{-u^{\varepsilon}\}_{\varepsilon}$, as $\varepsilon \rightarrow 0$,   $\Gamma$-converges to $-u$.
\end{prop}  A similar assertion holds for a lower semicontinuous function $v$, the sequence of inf-convolutions defined by  $v_\varepsilon(x) = -((-v)^\varepsilon(x))$.

\begin{proof}[Proof of Lemma \ref{dif:quo:lem}:] We will provide the full details for the first half i.e we establish \eqref{eq:diff-ineq-1}. The proof of \eqref{eq:diff-ineq-2} is very similar, and therefore omitted. The proof is divided into three steps. The first step is about simplifying the \eqref{eq:HJB-master}. We also mention that the main ideas of the proof are borrowed from \cite{Ishii:1995gs}.

\vspace{.2cm}
\noindent{\bf Step 1}: As per the given condition $u$ and $v$ respectively are sub- and supersolution of 

\begin{align*}
       & \sup_{\alpha\in \mathcal{A}} \inf_{\beta\in\mathcal{B}} \Big\{f^{\alpha, \beta}(x+\ell) + c^{\alpha,\beta}(x+\ell)u(x)+b^{\alpha,\beta}(x+\ell)\cdot\nabla u(x)\\&\hspace{7cm} +a^{\alpha,\beta}(x+\ell)  (-\Delta)^{\frac 12}u(x)\Big\} = 0,\\
  &\sup_{\alpha\in \mathcal{A}} \inf_{\beta\in\mathcal{B}} \Big\{f^{\alpha, \beta}(x) + c^{\alpha,\beta}(x)v(x)+b^{\alpha,\beta}(x)\cdot\nabla v(x) +a^{\alpha,\beta}(x)  (-\Delta)^{\frac 12}v(x)\Big\} = 0.
\end{align*} We now invoke Lemma \ref{lem:reduction-1} and claim that $u$ and $v$ respectively satisfy

\begin{align}
     \label{eq:reduction21}  & \sup_{\alpha\in \mathcal{A}} \inf_{\beta\in\mathcal{B}} \Big\{\tilde{f}^{\alpha, \beta}(x+\ell) + \tilde{c}^{\alpha,\beta}(x+\ell)u(x)+\tilde{b}^{\alpha,\beta}(x+\ell)\cdot\nabla u(x)\Big\} + (-\Delta)^{\frac 12}u(x) \le 0,\\
   \label{eq:reduction22} &\sup_{\alpha\in \mathcal{A}} \inf_{\beta\in\mathcal{B}} \Big\{\tilde{f}^{\alpha, \beta}(x) + \tilde{c}^{\alpha,\beta}(x)v(x)+\tilde{b}^{\alpha,\beta}(x)\cdot\nabla v(x)\Big\}  +  (-\Delta)^{\frac 12}v(x)\ge  0
\end{align} in the viscosity sense, where 
$\Big({\tilde{f}}^{\alpha, \beta}, {\tilde{c}}^{\alpha, \beta}, {\tilde{b}}^{\alpha, \beta}\Big) = \frac{1}{a^{\alpha, \beta} } \Big({f}^{\alpha, \beta}, {c}^{\alpha, \beta}, {b}^{\alpha, \beta}\Big). $ Now we choose $\lambda >0$ from \ref{A3}, define $\tilde{g}^{\alpha, \beta}_2(x) = \big(\tilde{c}^{\alpha,\beta}(x) -\lambda\big) v(x) +\tilde{f}^{\alpha, \beta}(x) $ and define $\tilde{g}^{\alpha, \beta}_1(x) = \big(\tilde{c}^{\alpha,\beta}(x+\ell) -\lambda\big) u(x) +\tilde{f}^{\alpha, \beta}(x+\ell) $. Then $||\tilde{g}^{\alpha,\beta}_i||_{W^{1,\infty}} < C$ for $i=1,~2$ and $(\alpha,\beta)\in \mathcal{A}\times \mathcal{B}$. Clearly, \eqref{eq:reduction21} and \eqref{eq:reduction22} imply that $u$ and $v$ satisfy

\begin{align}
     \label{eq:reduction23}   \sup_{\alpha\in \mathcal{A}} \inf_{\beta\in\mathcal{B}} \Big\{\tilde{g}_1^{\alpha, \beta}(x) +\tilde{b}^{\alpha,\beta}(x+\ell)\cdot\nabla u(x)\Big\} +\lambda u(x) + (-\Delta)^{\frac 12}u(x) \le 0,\\
   \label{eq:reduction24} \sup_{\alpha\in \mathcal{A}} \inf_{\beta\in\mathcal{B}} \Big\{\tilde{g}_2^{\alpha, \beta}(x) +\tilde{b}^{\alpha,\beta}(x)\cdot\nabla v(x)\Big\} +\lambda v(x) +  (-\Delta)^{\frac 12}v(x)\ge  0
\end{align} in the viscosity sense.

 \noindent{\bf Step 2}: In this step we prove that for every $\delta >0$, there is $\varepsilon_0> 0$ such that 
 
 \begin{align}
     \label{eq:reduction25}   \sup_{\alpha\in \mathcal{A}} \inf_{\beta\in\mathcal{B}} \Big\{\tilde{g}_1^{\alpha, \beta}(x) +\tilde{b}^{\alpha,\beta}(x+\ell)\cdot\nabla u^{\varepsilon}(x)\Big\} +\lambda u^{\varepsilon}(x) + (-\Delta)^{\frac 12}u^{\varepsilon}(x) \le \delta,\\
   \label{eq:reduction26} \sup_{\alpha\in \mathcal{A}} \inf_{\beta\in\mathcal{B}} \Big\{\tilde{g}_2^{\alpha, \beta}(x) +\tilde{b}^{\alpha,\beta}(x)\cdot\nabla v_{\varepsilon}(x)\Big\} +\lambda v_{\varepsilon}(x) +  (-\Delta)^{\frac 12}v_{\varepsilon}(x)\ge  -\delta
\end{align} whenever $\varepsilon < \varepsilon_0$. Note that $u(x)\le u^{\varepsilon}(x)$ and 
\begin{align*}
 u^{\varepsilon}(x) =\sup_{ y\in \Rn} \left[ u(x+y) - \frac{|y|^2}{\varepsilon} \right] 
                       = \sup_{ y\in \Rn} \left\{ u(x+y) - \frac{|y|^2}{\varepsilon}: |y|\le \sqrt{2M\varepsilon} \right\},
\end{align*} where $M = ||u||_{\infty}$. Note that $u(\cdot+y)$ solves 
 \begin{align}
    \label{eq:reduction27}  \sup_{\alpha\in \mathcal{A}} \inf_{\beta\in\mathcal{B}} \Big\{\tilde{g}_1^{\alpha, \beta}(x+y) +\tilde{b}^{\alpha,\beta}(x+y+\ell)\cdot\nabla U(x)\Big\} +\lambda U(x) + (-\Delta)^{\frac 12}U(x) \le 0
 \end{align}in the viscosity sense, and since $\lambda $ is positive, $u(\cdot+y)- \frac{|y|^2}{\varepsilon} $ also solves the same inequality \eqref{eq:reduction27}. Note that $\tilde{b}^{\alpha, \beta}(x), \,u(x)$~ and $g_1^{\alpha, \beta}(x)$ are Lipschitz continuous. Moreover, if $L$ is the Lipschitz constant
 of a generic  $U$, $U-\varphi$ has global maximum at $x$ then $|\nabla \varphi(x)|\le L.$ We use these facts in \eqref{eq:reduction27} and conclude that $U(x) = u(x+y)- \frac{|y|^2}{\varepsilon}  $ satisfies 
 \begin{align}
    \label{eq:reduction28}  \sup_{\alpha\in \mathcal{A}} \inf_{\beta\in\mathcal{B}} \Big\{\tilde{g}_1^{\alpha, \beta}(x) +\tilde{b}^{\alpha,\beta}(x+\ell)\cdot\nabla U(x)\Big\} +\lambda U(x) + (-\Delta)^{\frac 12}U(x) \le C \sqrt{2M\varepsilon}
 \end{align} in the viscosity sense, whenever $|y|\le  \sqrt{2M\varepsilon} $. We now choose $\varepsilon_0$ small enough such that $ C \sqrt{2M\varepsilon_0}< \delta$, then \eqref{eq:reduction28} implies that 
 $U(x) = u(x+y)- \frac{|y|^2}{\varepsilon}  $ satisfies 
 \begin{align}
    \label{eq:reduction28-1}  \sup_{\alpha\in \mathcal{A}} \inf_{\beta\in\mathcal{B}} \Big\{\tilde{g}_1^{\alpha, \beta}(x) +\tilde{b}^{\alpha,\beta}(x+\ell)\cdot\nabla U(x)\Big\} +\lambda U(x) + (-\Delta)^{\frac 12}U(x) \le \delta
 \end{align} in the viscosity sense. It is now straightforward to conclude, as in the proof of Theorem \ref{thm:exisence}, \newline that 
$ u^{\varepsilon} = \sup_{ y\in \Rn} \left\{ u(x+y) - \frac{|y|^2}{\varepsilon}: |y|\le \sqrt{2M\varepsilon} \right\}$ solves
\eqref{eq:reduction28-1} in the viscosity sense whenever $\varepsilon < \varepsilon_0$. The proof of \eqref{eq:reduction26} is similar.

 \vspace{.3cm}
  \noindent{\bf Step 3:} We now want to show that the conclusion from {Step 2} implies that there are constants $B_0$, $B_1$ and $A$ such that
  
       \begin{align*}
& -B_0\sup_{x\in \Rn}|u(x+\ell)-v(x)| -B_1 |\ell |+ \lambda (u^{\varepsilon}-v_{\varepsilon})\\&\hspace{5cm} - A|\nabla (u^{\varepsilon}-v_{\varepsilon}) |+(-\Delta)^{\frac 12}(u^{\varepsilon}-v_{\varepsilon}) \le 2\delta
       \end{align*} in the viscosity sense whenever $\varepsilon$ is small enough. Assume that  $\phi \in C^2(\Rn)$ such that $ (u^{\varepsilon}- v_{\varepsilon}) -\phi$ has global maximum at $x\in \Rn$, and 
$(u^{\varepsilon}- v_{\varepsilon})(x) = \phi(x)$. Note that $u^{\varepsilon}$ and $- v_{\varepsilon}$ are semiconvex functions, 
it means that for each of them there is a tangent paraboloid  touching it from below at every point $x\in \Rn$. 
In addition, $\phi$ touches $u^{\varepsilon}- v_{\varepsilon}$ from above at $x$ , then $u^{\varepsilon}$ 
and $- v_{\varepsilon}$ are $C^{1,1}$ at $x$ . Hence, the derivatives and fractional Laplacian of $u^{\varepsilon},  v_{\varepsilon}$
are well defined  at $x$.  Therefore, we finally have 
\begin{align*}
&-B_0\sup_{x\in\Rn}|u(x)+\ell)-v(x)|-B_1|\ell| \\&\qquad\qquad-A|\nabla(u^{\varepsilon}- v_{\varepsilon})| +\lambda (u^{\varepsilon}- v_{\varepsilon}) 
+ \laplas (u^{\varepsilon}- v_{\varepsilon}) \\
&\leq \sup_{\alpha\in \mathcal{A}} \inf_{\beta\in\mathcal{B}} \Big\{\tilde{g}_1^{\alpha, \beta}(x) +\tilde{b}^{\alpha,\beta}(x+\ell)\cdot\nabla u^{\varepsilon}(x)\Big\}  + \lambda u^{\varepsilon} + \laplas u^{\varepsilon} \\
&\qquad- \sup_{\alpha\in \mathcal{A}} \inf_{\beta\in\mathcal{B}} \Big\{\tilde{g}_2^{\alpha, \beta}(x) +\tilde{b}^{\alpha,\beta}(x)\cdot\nabla v_{\varepsilon}(x)\Big\}-\lambda v_{\varepsilon}(x) -  (-\Delta)^{\frac 12}v_{\varepsilon}(x)
\leq 2\delta.
\end{align*} In the end, we pass to the limit $\varepsilon \rightarrow 0$ and use Proposition \ref{prop1} and the stability of viscosity solutions under $\Gamma$-limit to conclude

         \begin{align*}
&-B_0\sup_{x\in\Rn}|u(x+\ell)-v(x)|-B_1|\ell| -A|\nabla(u- v)| +\lambda (u- v)\\
&\hspace{7cm}+ \laplas (u- v) \le 0.
\end{align*} 
      
      \end{proof} 
      
      As an immediate corollary to Lemma \ref{dif:quo:lem}, we have the following important fact. 
      
      \begin{cor}\label{cor:inequalities-difference} Let \ref{A1}-\ref{A4} be satisfied so that the viscosity solution $u$ of \eqref{eq:HJB-master} is Lipschitz continuous. Then there are constants $A$ and $B$ such that, for every given $\ell \in \Rn$, the function $v(x) =u(x+\ell)-u(x)$ satisfies
      
      \begin{align*}
      -A|\grad v| -B|\ell | + \lambda v + \laplas v \leq 0,\\
      A|\grad v| +B|\ell |  + \lambda v + \laplas v \geq 0
      \end{align*} in the viscosity sense.

            \end{cor}
       Note that if $u$ is the unique Lipschitz continuous viscosity solution of \eqref{eq:HJB-master}, then for $\ell \in S^{n-1}$ and $h>0$, the finite difference quotients
   
     $$\partial_{h, \ell} u(x) = \frac{u(x+h\ell)-u(x)}{h}$$
  are bounded and continuous functions. Moreover, by Corollary \ref{cor:inequalities-difference}, the difference quotient $v(x)=\partial_{h, \ell} u(x)$ satisfies 
  \begin{align}
 \label{eq:ineq-sub}  -A|\grad v| -B + \lambda v + \laplas v \leq 0,\\
 \label{eq:ineq-super}     A|\grad v| +B + \lambda v + \laplas v \geq 0
  \end{align} in the viscosity sense, where the constants $A$ and $B$ are independent of $h$ and $\ell$. 
      
        Our method requires introduction of an auxiliary time variable and convert \eqref{eq:ineq-sub} and \eqref{eq:ineq-super} into time dependent inequalities. This would enable us to use the machinery from \cite{Biswas2013, silvestre:2012} and prove H\"{o}lder continuity of $v$ in  \eqref{eq:ineq-sub} and \eqref{eq:ineq-super} . 
      
       \begin{lem}\label{time-conversion}
   Let $u$ be a bounded continuous function that satisfies 
\begin{align*}
-A|\grad u(x)| -B + \lambda u(x) + \laplas u(x) \leq 0 , \\
A|\grad u(x)| +B + \lambda u(x) + \laplas u(x) \geq 0 
\end{align*}
in the viscosity sense. Then, for $(t,x)\in (-\infty, 0]\times \R^n$, the function $v(t,x)= e^{\lambda t} u(x)$
satisfies the 
inequalities
\begin{align}\label{dif:quo1}
\partial_t v(t,x)-A|\grad v(t,x)| + \laplas v(t,x)-B \leq 0,\\
\label{dif:quo11}
\partial_t v(t,x)+A|\grad v(t,x)|+ \laplas v(t,x)+B \geq 0
\end{align} in the viscosity sense in $(-\infty, 0]\times \R^n$.

  \end{lem}
  \begin{proof}
  Let $\varphi\in C^{1,2}_b\big((-\infty,0]\times \Rn\big)$ be a test function such that $v(t,x)-\varphi(t,x)$ has a global maximum at $(t_0, x_0)$.
    Define $\psi(x)= e^{-\lambda t_0} \varphi(t_0, x)  $. Then 
    
    $$ u(x) -\psi(x) = e^{-\lambda t_0} \Big[ v(t_0, x) - \varphi(t_0, x)\Big].$$
    Therefore $u(x)-\psi(x)$ has a global maximum at $x_0$, which implies
    
    \begin{align*}
       0\ge& -A |\nabla \psi(x_0)|-B+\lambda u(x_0)+(-\Delta)^{\frac 12}\psi (x_0)\\
        =& -A e^{-\lambda t_0} |\nabla \varphi(t_0, x_0)| + e^{-\lambda t_0} (-\Delta)^{\frac 12}\varphi (t_0,x_0) + \lambda u(x_0)-B.
    \end{align*} In other words, we have  
    \begin{align}
  \label{eq:conversion-2}   -A |\nabla \varphi(t_0, x_0)| + (-\Delta)^{\frac 12}\varphi (t_0,x_0) + \lambda  e^{\lambda t_0} u(x_0)-B e^{\lambda t_0}\le 0.
    \end{align} Note that the function $e^{\lambda t} u(x_0) -\varphi(t, x_0)$ has global maximum at $t= t_0$, and hence $\partial_t \varphi(t_0, x_0) = \lambda e^{\lambda t_0} u(x_0)$. We now use this information in \eqref{eq:conversion-2} and obtain
    
     \begin{align}
\notag \partial_t \varphi(t_0, x_0)  -A |\nabla \varphi(t_0, x_0)| + (-\Delta)^{\frac 12}\varphi (t_0,x_0) -B\le 0.
    \end{align} This establishes that $v(t,x)$ satisfies \eqref{dif:quo1} in the viscosity sense. The proof that $v$ satisfies \eqref{dif:quo11} is similar.

  \end{proof}


 
  \section{Diminishing oscillation }\label{oscillation} 
  
  Note that proving H\"{o}lder continuity of $v$ in  \eqref{eq:ineq-sub}-\eqref{eq:ineq-super} is equivalent to proving the H\"{o}lder continuity of $v$ in \eqref{dif:quo1}-\eqref{dif:quo11}. We do this by using method of diminishing oscillation from \cite{silvestre:2010}, and the necessary details are provided below.  
  
  \begin{prop}\label{thm:diminshing_oscillation-1}
Let $u$ be an upper semicontinuous function such that $u\le 1$ in $[-2,0]\times \R^n$ .  Also, $u$ satisfies
\begin{align}
\label{eq:osc_sub_sol}
u_t - A |\nabla u| + (-\Delta)^{\frac{1}{2}} u \le \vartheta_0,
\end{align} interpreted in the viscosity sense, in $[-2,0]\times B_{2+2A}$. Assume further that there is a $\mu > 0$ such that
\begin{align*}
|\{u \le 0\}\cap [-2,-1]\times B_1| \ge \mu.
\end{align*}
 Then, for sufficiently small $\vartheta_0$, there is a $\theta \in (0,1)$  such that $ u\le 1-\theta $ in $[-1,0]\times B_1$. (The maximal value of $\theta$ depends on $A,~ \mu$ and $n$.)
\end{prop}
     \begin{proof}
         This proposition is the backbone of our method, and the proof is readily available in the literature. The interested reader can consult \cite{Biswas2013} or \cite{silvestre:2010}, where the proof is rigorously detailed. 
     \end{proof}
  
  We now apply Proposition \ref{thm:diminshing_oscillation-1} to establish the most crucial technical assertion of this paper. 

 \begin{theorem}[diminishing oscillation]\label{thm:diminishing_oscilation}
 Let $\xi$ be a bounded continuous  function satisfying the inequalities
 \begin{align*}
  & \xi_t -A|\nabla \xi | + (-\Delta)^{\frac{1}{2}} \xi \le \varepsilon\\
   &\xi_t + A|\nabla \xi|+(-\Delta)^{\frac{1}{2}} \xi \ge -\varepsilon
 \end{align*} in the viscosity sense in $Q_1 = [-1,0]\times B_1$. There are universal constants $\theta \in (0,1)$, $\alpha_0 >0$ and $\varepsilon_0$ (depending only on $A$ and $n$) such that if $\varepsilon\le \varepsilon_0 $ and 
 \begin{align*}
 |\xi| &\le 1\quad &\text{in} \quad Q_1 = [-1,0]\times B_1,\\
|\xi| &\le 2 |(4+4A)x|^\alpha -1 \quad &\text{in}\quad  [-1,0]\times B_1^c,
 \end{align*} for $\alpha \le \alpha_0$, then
 \begin{align*}
\underset{Q_{1/(4+4A)}}{\mbox{osc}}\xi \le 2(1-\theta).
 \end{align*}
 \end{theorem}
 \begin{proof} First, we consider the rescaled version of $\xi$ as
    \begin{align*}
       \tilde{\xi} = \xi\big(t/R, x/R\big),\quad \text{where}\quad R = 4+4A. 
    \end{align*} Clearly, $\tilde{\xi}$ satisfies the inequalities 
     \begin{align*}
  &\tilde{\xi}_t -A|\nabla \tilde{\xi} | + (-\Delta)^{\frac{1}{2}} \tilde{\xi}\le \frac{\varepsilon}{R}\\
    &\tilde{\xi}_t +A|\nabla \tilde{\xi} | + (-\Delta)^{\frac{1}{2}} \tilde{\xi}\ge - \frac{\varepsilon}{R},
 \end{align*} in the viscosity sense. Furthermore, one must have either  $ |\{\tilde{\xi} \le 0\} \cap ([-2,-1]\times B_1)| \ge \frac{|B_1|}{2}$ or $  |\{\tilde{\xi} \ge 0\} \cap ([-2,-1]\times B_1)| \ge \frac{|B_1|}{2}$. We simply assume the former or else we estimate the oscillation of $-\tilde{\xi}$ instead. At this point, if Proposition \ref{thm:diminshing_oscillation-1} was directly applicable to $\tilde{\xi}$ then we will have 
\begin{align*}
  \sup_{(t,x)\in Q_1} \tilde{\xi}(t,x) \le 1-\theta,
\end{align*} which implies $\underset{Q_{1/R}}{\mbox{osc}}\, \xi \le 2(1-\frac{\theta}{2})$, thereby proving the theorem. However, the condition that $\tilde{\xi}$ needed to be bounded above by $1$ is unclear.  To this end, we simply define $u=\min(1,\tilde{\xi})$ and estimate the inequalities satisfied by $u$ instead. 

 Clearly, $\tilde{\xi} \le 1$ in $Q_R$ and $\tilde{\xi} =  u$ in $Q_R$. For a given test function $\varphi$, let $u-\varphi$ has a global maximum at $(t,x)\in [-2,0]\times B_{2+2A}$ and $u(t,x)= \varphi(t,x)$. Keep in mind that $(t,x)$ may not be a point of global maximum for $\tilde{\xi}-\varphi$, but we modify $\varphi$ outside $[-(2+2A),0]\times B_{3+2A}(x)$ so that $(t,x)$ becomes a point of global maximum for $\tilde{\xi}-\varphi$. Therefore, we have
 \begin{align}
\label{eq:osci-measure} \partial_t \varphi(t,x) -A|\nabla \varphi(t,x)| +\mathcal{I}_\kappa \varphi(t,x) +\mathcal{I}^\kappa \tilde{\xi}(t,x) \le \frac{\varepsilon}{R}.
  \end{align} At this point choose $\varepsilon_0 < R\, \vartheta_0 $. It is easy to see that, for every $\kappa\in (0,1)$
  
  \begin{align*}
     \mathcal{I}^\kappa u(t,x)-  \mathcal{I}^\kappa\tilde{\xi}(t,x) \le & c(n,1/2)\int_{x+y\in B_{R}^c}\big(\tilde{\xi}(t,x+y) -\min(1, \tilde{\xi}(t,x+y))\big)\, \frac{dy}{|y|^{n+1}}\\
     &  =c(n,1/2)\int_{x+y\in B_{R}^c}\big(\tilde{\xi}(t,x+y) -1\big)^+\, \frac{dy}{|y|^{n+1}}\\
     & \le C \int_{x+y\in B_{R}^c}2 \big(|R^2 (x+y)|^\alpha -1\big)\, \frac{dy}{|y|^{n+1}}\\
     & \le \vartheta_0 -\frac{\varepsilon}{R}
  \end{align*} if $\alpha$ is chosen to be small enough. Note that the choice of $\alpha$ is independent of $\kappa$ and $(t,x)$.  Therefore, we now combine this with \eqref{eq:osci-measure} and conclude
   \begin{align}
\notag \partial_t \varphi(t,x) -A|\nabla \varphi(t,x)| +\mathcal{I}_\kappa \varphi(t,x) +\mathcal{I}^\kappa u(t,x) \le\vartheta_0.
  \end{align} In other words, Proposition \ref{thm:diminshing_oscillation-1} applies to $u$ and hence the proof follows.

 \end{proof}

  \section{Regularity estimate: the endgame}\label{holder-regularity}
 We now apply Theorem \ref{thm:diminishing_oscilation} to prove the all-important H\"{o}lder continuity result for the difference quotients. 
  
 \begin{prop}\label{regularity-derivative}
Let $u$ be a bounded continuous function on $\Rn$ that satisfies 
\begin{align} \label{dif:quo}
-A|\grad u| -B + \lambda u + \laplas u \leq 0, \\
A|\grad u| +B  + \lambda u + \laplas u \geq 0 \label{dif:quotient}
\end{align}
in the viscosity sense. Then there exists a positive constant $\sigma \in (0,1)$ such that $u$ is H\"{o}lder continuous of
 exponent $\sigma$. Moreover, there are constants $C$ and $K$(depending on $A$ ,$B$ and $\lambda$) such that 
\begin{align*}
|u(x)- u(y)| \leq C(||u||_{L^{\infty}}+K) |x-y|^{\sigma}  \quad \text{for all}\quad x,y\in \Rn.
\end{align*} 
\end{prop}

 \begin{proof}[Proof of Proposition  \ref{regularity-derivative}] 

 Fix $x_0\in \Rn$ and define
 \begin{align*}
  \xi(t,x)= \frac{e^{\lambda t}u(x+x_0)}{ 2 ||u||_{L^\infty} + 2}; \quad (t,x)\in [-2,0]\times \Rn;
 \end{align*} where $\lambda$ is the given positive constant. Clearly H\"{o}lder continuity of $\xi$ at $(0,0)$ would mean H\"{o}lder continuity of $u$ at $x_0$.  At this point we recall that $Q_s= [-s,0]\times B_s$ where $s> 0$ and $B_s$ ball of radius $s$ around $0\in \Rn$. Furthermore, fix $r = \frac{1}{4+4A}$.
 
 We intend to establish that $\xi$ is  H\"{older} continuous at $(0,0)$, and we do so by establishing that there is $\sigma \in (0,1)$ such that
 \begin{align}
    \label{eq:holder-conversion}\underset{Q_{r^k}}{\mbox{osc}}\,\xi\, \le 2 r^{\sigma k}\quad \text{for every}~ k\in \mathbb{N}.
 \end{align} Now establishing \eqref{eq:holder-conversion} is equivalent to finding out a sequence of nested intervals $(a_k, b_k)$, k=0, 1,2....; such that for all $(t,x)\in Q_{r^k}$
 \begin{align}
    \label{eq:holder-conversion-1}  a_k \le \xi(t,x) \le b_k \quad\text{and}\quad b_k-a_k =2r^{\sigma k}.
 \end{align} To this end, we first fix $\sigma_0\in (0,1)$ and find $k_0\in \mathbb{N}$ large enough so that $\frac{B\, r^{(1-\sigma_0)k_0}}{2}\le \varepsilon_0$ where $\varepsilon_0$ is given by Theorem \ref{thm:diminishing_oscilation}. We will now find a $\sigma < \sigma_0$ and a nested sequence $(a_k, b_k)$ such that \eqref{eq:holder-conversion-1} folds. We will make our selection in two steps. In the first step, we find $(a_k, b_k)$ for $k= 0, 1,......., k_0$ and the rest in the second step.
 
       \vspace{.3cm}
 \noindent{\bf Step 1.} Choose $\sigma < \sigma_0$ small enough such that $ \frac 12\le r^{\sigma k_0} < 1$.
 Note that $||\xi(t,x)||_{L^\infty} \le \frac 12$, therefore $\underset{Q_{1}}{\mbox{osc}}\,\xi \le 1\le  2 r^{\sigma k_0} $.
 It is enough to find a  $ (a_{k_0}, b_{k_0})$ such that 
 $$b_{k_0}-a_{k_0} = 2 r^{\sigma k_0}\quad \text{and}\quad a_{k_0}\le \xi(t,x)\le b_{k_0} \quad \text{for all}\quad (t,x)\in Q_1.$$ It is now trivial to find out constants $a_0\le a_1\le\cdots\le a_{k_0-1}\le a_{k_0} $ and $b_0\ge b_1\ge  b_2\ge\cdots \ge b_{k_0-1}\ge b_{k_0} $ such that $b_j - a_j = 2 r^{\sigma j}$ for $j = 0, 1,....., k_0$. Clearly, it also holds that $a_j \le \xi(t,x)\le b_j$ for all $(t,x)\in Q_1\supseteq Q_{r^j}$.
 
     \vspace{.2cm}
  \noindent{\bf Step 2.} We now complete our choice of $(a_j, b_j)$ for $j= k_0+1, k_0+2, k_0+ 3,.......$ satisfying \eqref{eq:holder-conversion-1}.  We use method of induction and assume that we already have $(a_j, b_j)$ for $j$ up to some $ k\ge k_0$.  We want to establish the existence of 
  $(a_{k+1}, b_{k+1})$.  At this point, we define 
      $$ \xi_k (t,x) = \Big(\xi(r^k t, r^k x) -\frac{a_k+b_k}{2}\Big) r^{-\sigma k}.$$
      
      Clearly, by Lemma \ref{time-conversion}, $\xi_k$ satisfies 
      
      \begin{align}\label{dif:quo1.1}
\partial_t \xi_k(t,x)-A|\grad \xi_k(t,x)| + \laplas \xi_k(t,x)\le  \frac{B r^{(1-\sigma)k}}{2+2||u||_{\infty}} \le  \frac{B r^{(1-\sigma_0)k_0}}{2} \le \varepsilon_0  ,\\
\label{dif:quo1.2}
\partial_t \xi_k(t,x)+A|\grad \xi_k(t,x)|+ \laplas \xi_k(t,x) \geq - \frac{B r^{(1-\sigma)k}}{2+2||u||_{\infty}} \ge -  \frac{B r^{(1-\sigma_0)k_0}}{2} \ge -\varepsilon_0,
\end{align}in the viscosity sense as $k \ge k_0$ and $\sigma <\sigma_0$. 
  
  \vspace{.2cm}
\noindent{{\bf Claim}}: It holds that 
\begin{align*}
 &|\xi_k(t,x)| \le 1 \quad \text{in}\quad Q_1, 
\end{align*} and 
\begin{align}
\label{eq:holder_oscilllation}|\xi_k(t,x)| \le 2|r^{-1}x|^{\sigma }-1 \quad\quad\text{if}\quad\quad |x| >1.
\end{align}

\noindent{\it  Justification}: The first part is an immediate consequence of the hypothesis on $(a_k, b_k)$. The second part is argued as follows: let $m\in \{1,2........, k\}$ be an integer and $(t,x)\in Q_{r^{-m}}$. Then $(r^k t, r^k x)\in Q_{r^{k-m}}$, and 
\begin{align*}
  \xi_k(t,x) &= \big(\xi(r^k t, r^k x)-\frac{b_k-a_k}{2} -a_k\big)r^{-\sigma k}\\
                  & =  \big(\xi(r^k t, r^k x) -a_k\big)r^{-\sigma k} -1\\
                  & \le  \big(\xi(r^k t, r^k x) -a_{k-m}\big)r^{-\sigma k} -1\\
                  &\le 2 r^{\sigma(k-m)} r^{-\sigma k} -1 = 2r^{-\sigma m}-1.
\end{align*} In addition
   
\begin{align*}
 - \xi_k(t,x) &= -\big(\xi(r^k t, r^k x)-\frac{a_k-b_k}{2} -b_k\big)r^{-\sigma k}\\
                  & =  \big(b_k-\xi(r^k t, r^k x)\big)r^{-\sigma k} -1\\
                  & \le  \big(b_{k-m}-\xi(r^k t, r^k x) \big)r^{-\sigma k} -1\\
                  &\le 2 r^{\sigma(k-m)} r^{-\sigma k} -1 = 2r^{-\sigma m}-1.
\end{align*} Therefore 
\begin{align}
 \label{eq:revision-new-2.1}   | \xi_k(t,x)|\le 2 r^{-\sigma m} -1\quad \text{if}\quad (t,x)\in Q_{r^{-m}}, ~ m\in \{1,2........, k\}.
\end{align}  Moreover, it also holds that 

 \begin{align} 
  \label{eq:revision-new-3.1}  \quad \quad& |\xi_k(t,x)| \le 2r^{-\sigma k}-1 \quad \text{in}\quad [-1,0]\times \R^n.
\end{align} Therefore, from \eqref{eq:revision-new-3.1} we conclude
\begin{align}
\label{eq:holder_oscilllation_1}|\xi_k(t,x)| \le 2|r^{-1}x|^{\sigma }-1 \quad\quad\text{if}\quad\quad |x| \ge r^{-(k-1)}.
\end{align} We now simply  combine \eqref{eq:revision-new-2.1} and \eqref{eq:holder_oscilllation_1} and complete the justification.
\vspace{.2cm}
Therefore, by Theorem \ref{thm:diminishing_oscilation}, there are universal constants $\alpha_0\in (0,1)$ and $\theta\in (0,1)$ (depending on $n, A$) such that if \eqref{eq:holder_oscilllation}
holds for any $0<\sigma < \alpha_0$, then
 \begin{align}
 \label{eq:revision-new-2} \text{osc}_{Q_r} \xi_k \le 2(1-\theta).
 \end{align} We now choose $\sigma \le \min(\alpha_0, \sigma_0)$ such that $(1-\theta) < r^{\sigma}$.  For this choice of $\sigma$, we must have
 \begin{align}
\label{eq:final-estimate} \underset{Q_{r}}{\mbox{osc}}\,\xi_k(t,x) \le 2r^{\sigma}, \quad \text{in other words},\quad  \underset{Q_{r^{k+1}}}{\mbox{osc}}\,\xi(t,x) \le 2r^{(k+1)\sigma}.
 \end{align} Hence the interval $(a_{k+1}, b_{k+1})$ could also be chosen. This completes the proof.

  \end{proof}
 
  \begin{proof}[Proof of Theorem \ref{main-theorem}]
    Let  $u$ be the unique viscosity solution of \eqref{eq:HJB-master} and $\lambda$ in \ref{A3} is large enough so that $u$ is Lipschitz continuous. Fix a unit vector $\ell \in S^{n-1}$ and let  $h> 0$ be a positive constant. The quantity

 \begin{align*}
 \partial_{h,\ell}u(x) = \frac{u(x+h\ell)-u(x)}{h}
 \end{align*} is a difference quotient of $u$ along the vector $\ell$. For every fixed ordered pair $(h,\ell)$, the function $\partial_{h,\ell}u(x)$ is continuous and bounded above by the Lipschitz constant of $u$. Clearly, by Corollary \ref{cor:inequalities-difference}, the function $v= \partial_{h,\ell}u(x)$ satisfies the inequalities \eqref{dif:quo} and \eqref{dif:quotient} in the viscosity sense. In other words, by Proposition \ref{regularity-derivative}, there is $\sigma \in (0,1)$ such that
  \begin{align*}
   ||\partial_{h,\ell}u(x)||_{C^{0,\sigma}( \R^n)} \le C(||\nabla u ||_{L^\infty} + K)
  \end{align*} for every  $h> 0$ and $\ell\in S^{n-1}$. We now apply Arzela-Ascoli's theorem and pass to the limit $h\rightarrow 0$ and conclude that $\partial_{\ell} u(x)$ exists and
   \begin{align*}
   ||\partial_{\ell}u(x)||_{C^{0,\sigma}( \R^n)} \le C(||\nabla u ||_{L^\infty} + K)
  \end{align*} for all unit vectors $\ell\in \R^n $. In other words

    \begin{align*}
   ||\nabla u (x)||_{C^{0, \sigma}( \R^n)} \le C(||\nabla u||_{L^\infty} + K), 
  \end{align*} and this concludes the proof in view of Theorem \ref{thm:lipschitz}.
 \end{proof}

\appendix

\section{Proof of existence and Lipschitz continuity}\label{appendix}

   We begin this section with the proof of existence of viscosity solutions to \eqref{eq:HJB-master}, and as mentioned, the proof is classical and uses Perron's method.  
   
   \begin{proof}[Proof of Theorem \ref{thm:exisence} ]
   
    Note that the functions $\overline{u}(x) = \frac{M}{\lambda}$ and $\underline{u}(x) =\frac{-M}{\lambda}$ are respectively super- and subsolution of \eqref{eq:HJB-master}, where the constant $M$ is defined by $M =  \sup_{\A,\B}\sup_{ x\in \Rn}  |f^{\A,\B}(x)|.  $ This means, by comparison principle, any viscosity solution has to be bounded. Now define
   $$ v(x) = \sup \{ w(x) : w\leq \overline{u} , w \ \mbox{is a subsolution of} \ \eqref{eq:HJB-master}\}. $$
   Furthermore, let $v^*$ and $v_*$ denote the upper and lower semicontinuous envelopes of $v$, defined by 
 $$ v^*(x) = \displaystyle \lim_{r\rightarrow 0 } \sup \{ v(y) : y\in B_r(x) \}\,~ \mbox{and} \ 
 v_*(x) = - (-v)^*(x) . $$ 
    Note that \begin{align*}
v^*(x) & = \displaystyle \lim_{r\downarrow 0 } \sup \{ v(y) : y\in B_r(x) \} 
 \geq \displaystyle \lim_{r\downarrow 0 } \inf \{ v(y) : y\in B_r(x) \} \\
& = - \displaystyle \lim_{r\downarrow 0 } \sup \{ - v(y) : y\in B_r(x) \} 
 = -(-v)^*(x) = v_*(x).
\end{align*} Moreover, it is trivially seen that $\underline{u}= (\underline{u})_* \leq v_* \leq v^* \leq (\overline{u})^* = \overline{u}. $ We now claim that $v^*$ is a subsolution of \eqref{eq:HJB-master}, a justification could be given as follows. Note that for every $x\in \Rn$, there is a sequence $\{v_m\}$ of subsolutions of \eqref{eq:HJB-master} such that $\lim_{m\rightarrow\infty } v_m(y_m) = v^*(x)$, where $\{y_m\}$ is a sequence converging to $x$.

  For a test function $\phi \in C^2_b(\Rn)$ , if $v^* - \phi$ has a strict global maximum at $x$. Therefore there is a sequence
 $\{x_m\}$ such that $v_m - \phi$ have maximum at $x_m$ and 
 $$\displaystyle \lim_{m\rightarrow \infty} (x_m, v_m(x_m)) = (x, v^*(x)).  $$
By definition of viscosity subsolution we get, for every $ m \in \mathbb{N}$ 
\begin{align*} &\sup_{\A \in \SA} \inf_{\B \in \SB} \{ f^{\A , \B} (x_m) + c^{\A , \B} (x_m) v_m(x_m) + 
b^{\A , \B} (x_m)\cdot \grad \phi(x_m) \\&\hspace{6cm}+ a^{\A , \B} (x_m) \laplas \phi(x_m) \} \leq 0.
\end{align*}
By letting $m \rightarrow \infty$,
$$ \displaystyle \sup_{\A \in \SA} \inf_{\B \in \SB} \{ f^{\A , \B} (x) + c^{\A , \B} (x) v^*(x) + 
b^{\A , \B} (x)\cdot\grad \phi(x) + a^{\A , \B} (x) \laplas \phi(x) \} \leq 0. $$
Hence, $v^*$ is a subsolution of \eqref{eq:HJB-master}. 

Next, we claim that $v_*$ is a supersolution of \eqref{eq:HJB-master} and the proof is done by contradiction. 
Assume for the moment that $v_*$ is not a supersolution. Then there exist $y \in \Rn$ and $\phi \in C^2_b(\Rn)$ such that 
$v_*(y) =\phi(y)$ and $v_* -\phi $ has global minimum at $y$, but
      \begin{align} \label{exist:1} 
\displaystyle \sup_{\A \in \SA} \inf_{\B \in \SB} \{ f^{\A , \B} (y) + c^{\A , \B} (y) \phi(y) + 
b^{\A , \B} (y)\cdot\grad \phi(y) + a^{\A , \B} (y) \laplas \phi(y) \} < 0. 
\end{align}
  Note that $v_* \leq \overline{u}$,  but at the point $y$ we claim that  $v_*(y) < \overline{u}(y)$. Otherwise, $\phi(y)= v_*(y) = \overline{u}(y)$ and therefore  $\overline{u}- \phi $ has global minimum at $y$. This means 
$$ \displaystyle \sup_{\A \in \SA} \inf_{\B \in \SB} \{ f^{\A , \B} (y) + c^{\A , \B} (y) \phi(y) + 
b^{\A , \B} (y)\cdot\grad \phi(y) + a^{\A , \B} (y) \laplas \phi(y) \} \geq 0;$$
contradicting  \eqref{exist:1}.

 Now by continuity of $\phi$ and $\overline{u}$, there 
exist $\gamma_1 >0$ and $\delta_1>0$ such that 
$$\phi + \gamma_1 \leq \overline{u} \quad \mbox{in} ~B_{\delta_1} (y).$$

Moreover, form \eqref{exist:1} we see that there exist $\gamma_2>0$ and $\delta_2>0$ such that 
\begin{align} \label{exist:2}
 \displaystyle \sup_{\A \in \SA} \inf_{\B \in \SB} \{ f^{\A , \B} (x) + c^{\A , \B} (x) (\phi(x)+\gamma) + 
b^{\A , \B} (x)\cdot\grad \phi(x) + a^{\A , \B} (x) \laplas \phi(x) \} \leq 0
\end{align} whenever $\gamma\leq 
\gamma_2$ and $ x \in B_{\delta_2}(y) $.

In addition, as $v_* - \phi$ has strict minimum at $y$, there exist $ \gamma_3>0$ and $0< \delta \leq \min \{\delta_1
,\delta_2\}$ such that $v_* - \phi >\gamma_3 \quad \mbox{on} \ \ \partial B_{\delta}(y).$ Now set $\gamma = \min   
\{\gamma_1, \gamma_2, \gamma_3\}$ and define 
\begin{equation*}
    w = \begin{cases}
               \max\{\phi+\gamma;v^* \}               & \text{on}~ \ B_{\delta}(y)\\
               v^*               & \mbox{otherwise}. 
           \end{cases}
\end{equation*}
   We will show that $w$ is a viscosity subsolution of \eqref{eq:HJB-master}. Let $x\in \Rn$ and $\psi \in C_b^2(\Rn)	$ be a test 
function such that $\psi(x) = w(x)$ and $w - \psi$ has maximum at $x$. Depending on whether $w=v^*$ or $w= \phi + \gamma $ at $x$, either $v^* - \psi$ or $\phi+\gamma -\psi $ has 
global maximum at $x$.  If $w= v^*$ at $x$ then as $v^*$ is a viscosity subsolution of \eqref{eq:HJB-master}, it follows 
$$ \displaystyle \sup_{\A \in \SA} \inf_{\B \in \SB} \{ f^{\A , \B} (x) + c^{\A , \B} (x) \psi(x) + 
b^{\A , \B} (x)\cdot\grad \psi(x) + a^{\A , \B} (x) \laplas \psi(x) \} \leq 0.$$ 
   
   Otherwise, $\grad \phi(x) = \grad \psi(x) $ and $\laplas \phi(x) \geq \laplas \psi(x). $ Hence form \eqref{exist:2}
we get
$$ \displaystyle \sup_{\A \in \SA} \inf_{\B \in \SB} \{ f^{\A , \B} (x) + c^{\A , \B} (x) \psi(x) + 
b^{\A , \B} (x)\cdot\grad \psi(x) + a^{\A , \B} (x) \laplas \psi(x) \} \leq 0.$$ 
Hence $w$ is a viscosity subsolution of \eqref{eq:HJB-master}. Finally,  $w_*(y) \geq (\phi+\gamma)_*(y) = \phi(y) + \gamma = v_*(y) + \gamma $. This means that  there is $ z\in B_{\delta}(y)$ such that $w(z) > v(z)$ . 
It contradicts the definition of $v$. Hence $v_*$ is a viscosity supersolution of \eqref{eq:HJB-master}.
 
Then by Theorem \ref{comparison} we get $ v^* \leq v_*$; it follows $v^*= v_* = v$ and hence $v$ is a viscosity 
solution of the equation \eqref{eq:HJB-master}.

   \end{proof}

We now prove Theorem \ref{thm:lipschitz}, and this is done by using the method of doubling the variables. We begin the following observation. 

\begin{lem} \label{lem.lip1}
Let $u, -v \in USC_b(\Rn)$ be respectively a viscosity subsolution and a supersolution of \eqref{eq:HJB-master}. Suppose $\phi \in C^2(\Rn\times \Rn)$ and $(x_0, y_0) \in \Rn\times \Rn$ be such that $u(x) -v(y) -\phi(x,y)$ has a global maximum at $(x_0,y_0)$. Then, for all $\delta \in (0,1)$, 
\begin{align*}
& \displaystyle \sup_{\A \in \SA} \inf_{\beta \in \SB}\big\{ f^{\A , \beta} (x_0) + c^{\A , \beta} (x_0) u(x_0) + 
b^{\A , \beta} (x_0)\cdot D_x \phi(x_0,y_0)\\&\hspace{6cm} + a^{\A , \beta} (x_0) [\laplask \phi_1(x_0) + \laplasK u(x_0)] \big\} \leq 0\\
\mbox{and } & \displaystyle \sup_{\A \in \SA} \inf_{\beta \in \SB}\big\{ f^{\A , \beta} (y_0) + c^{\A , \beta} (y_0) v(y_0) + 
b^{\A , \beta} (y_0)\cdot(-D_y \phi(x_0,y_0))\\&\hspace{5cm} + a^{\A , \beta} (y_0) [\laplask \phi_2(y_0) + \laplasK v(y_0)]\big\} \geq 0,
\end{align*} where $\phi_1(\cdot) = \phi(\cdot, y_0)$ and $\phi_2(\cdot)=- \phi(x_0,\cdot)$. 
\end{lem}
 The proof is trivial and follows directly from the definition of viscosity solution.

\begin{rem}
      In view of Lemma \ref{lem:reduction-1}, the strict ellipticity assumption in \ref{A4} permits us
      to rewrite the equation \eqref{eq:HJB-master} into a form which is similar but additionally 
      has $a^{\A, \B}(x) \equiv 1$. Therefore, while proving Lipschitz continuity of the viscosity solution,
      it is enough to assume $a^{\A, \B}(x) \equiv 1$. This result could be appropriately translated for \eqref{eq:HJB-master} if \ref{A4} holds. However, it is important to mention here that the Lipschitz continuity holds even for `degenerate' equations where $a^{\A,\B}(x)$ may vanish at certain points i.e \ref{A4} fails to hold. In such a scenario, our proposed transformation does not work. 
\end{rem}

\begin{proof}[Proof of Theorem \ref{thm:lipschitz}.] For positive constants $\gamma$ and $\varepsilon$, define 
\begin{align*}
 &\phi(x,y) = \dfrac{\gamma}{2} |x-y|^2 + \dfrac{\varepsilon}{2} (|x|^2+|y|^2)\quad
  \mbox{and} \quad \psi(x,y) = u(x) -u(y) -\phi(x,y) .
\end{align*} In view of the above remark, from now onward, we will simply assume that $a^{\A, \B}\equiv 1$. It is trivially seen that the function $\psi(x,y)$ is bounded above and there is $(x_0, y_0)\in \Rn\times \Rn$ such that $M_\varepsilon := \sup_{x,y \in \Rn} \psi(x,y) = \psi(x_0, y_0)$.  It now follows by Lemma \ref{lem.lip1} that 
\begin{align*}
&\displaystyle \sup_{\A \in \SA} \inf_{\B \in \SB} \{ f^{\A , \B} (x_0) + c^{\A , \B} (x_0) u(x_0) + 
b^{\A , \B} (x_0)\cdot D_x \phi(x_0,y_0) +  \laplask \phi_1(x_0) + \laplasK u(x_0) \}\\
&- \displaystyle \sup_{\A \in \SA} \inf_{\B \in \SB} \{ f^{\A , \B} (y_0) + c^{\A , \B} (y_0) u(y_0) + 
b^{\A , \B} (y_0)\cdot (-D_y \phi(x_0,y_0)) \\&\hspace{8cm}+ \laplask \phi_2(y_0) + \laplasK u(y_0) \} \leq 0 .
\end{align*}

Note that $\psi(0,0) = 0$, therefore $M_\varepsilon \ge 0$ and this implies that $u(x_0)-u(y_0) \ge 0$. Therefore
\begin{align} \label{lip.cont0}
\notag &\lambda (u(x_0) - u(y_0) ) \\ \leq  \notag &
\displaystyle \sup_{\A \in \SA} \inf_{\B \in \SB} \{ f^{\A , \B} (y_0) + c^{\A , \B} (y_0) u(y_0) + 
b^{\A , \B} (y_0)\cdot (-D_y \phi(x_0,y_0)) \\\notag & \hspace{7.5cm}+ [\laplask \phi_2(y_0) + \laplasK u(y_0)] \} \\
\notag& \hspace{.3cm}- \displaystyle \sup_{\A \in \SA} \inf_{\B \in \SB} \{ f^{\A , \B} (x_0) + c^{\A , \B} (x_0) u(y_0) + 
b^{\A , \B} (x_0)\cdot D_x \phi(x_0,y_0)\\&\hspace{7.5cm} +[\laplask \phi_1(x_0) + \laplasK u(x_0)] \},
\end{align} where $\lambda$ is the constant in \ref{A3}. We now need to find a suitable upper bound for the RHS of \eqref{lip.cont0}. Thanks to \ref{A1}-\ref{A2}, it follows from straightforward computation that there exists constants $K_0, K_1, B$ depending on \ref{A1}-\ref{A2} and $||u||_{L^\infty}$ such that 
\begin{align} \label{lip.cont1}
\notag & f^{\A , \B} (y_0) + c^{\A , \B} (y_0) u(y_0) + 
b^{\A , \B} (y_0)\cdot (-D_y \phi(x_0,y_0))\\ \notag
& \hspace*{3cm}-[f^{\A , \B} (x_0) + c^{\A , \B} (x_0) u(y_0) + 
b^{\A , \B} (x_0)\cdot D_x \phi(x_0,y_0)] \\ 
& \leq K |x_0 -y_0| + K_1 \gamma |x_0 - y_0|^2 + B \varepsilon (1+ |x_0|^2 + |y_0|^2) .
\end{align} for all $(\A, \B)$.

We now estimate the nonlocal terms, keeping in mind the simplifying assumption that we have made at the beginning of the proof. Note that 
\begin{align*}
& \laplask \phi_2(\yo) - \laplasK \phi_1(\xo)\\ =~& c(n,1/2) \int_{B_{\delta}} \dfrac{1}{|z|^{n+1}}\Big[\phi(\xo,\yo+z) 
     -2\phi(\xo,\yo) + \phi(\xo +z,\yo)\\ &\hspace{4cm}+ z\cdot(\gamma(\xo-\yo)-\varepsilon \yo)  - z\cdot(\gamma(\xo-\yo) +\varepsilon \xo) \Big]\, dz \\
     =&  c(n,1/2) \int_{B_{\delta}} \frac{(\gamma +\varepsilon)}{|z|^{n-1}}\,dz \\
     =& c(n,1/2) \omega(n) (\varepsilon+\gamma)\delta,
     \end{align*} where $\omega(n)$ is the surface area of $S^{n-1}$. Note that $\mathcal{I}^{\delta}u$ could be written as $\mathcal{I}^{\delta}u(x):= \mathcal{I}^{\delta}_1u(x)+\mathcal{I}^{\delta}_1u(x)$, where 
      \begin{align*}
     \mathcal{I}_1^{\delta}u(x) =c(n,1/2)  \int_{\delta <|z|<1 } \dfrac{u(x) - u(x +z)}
{|z|^{n+1}} \ dz, \\ \mathcal{I}_2^{\delta}u(x)= c(n,1/2)\int_{|z|>1} \dfrac{u(x)  -u(x+z)}{|z|^{n+1}} \,dz. 
\end{align*}
Moreover, for any $\ell \in \Rn$, it holds that $$  \mathcal{I}_1^{\delta}u(x) = c(n,1/2) \int_{\delta <|z|<1 } \dfrac{u(x) - u(x +z) -\ell\cdot z}
{|z|^{n+1}} dz.$$  

With this notation, we now have  \\$ \laplasK u(\yo) - \laplasK u(\xo) =  \laplasK_{1} u(\yo) - \laplasK_1 u(\xo) + \laplasK_2 u(\yo) - \laplasK_2 u(\xo) $ and 
\begin{align*}
&  \laplasK_{1} u(\yo) - \laplasK_{1} u(\xo)\\ = &~c(n,1/2)\int_{\delta<|z|<1}\dfrac{ u(\yo)-u(\xo) -u(\yo+z) +u(\xo+z) - \varepsilon(\xo+\yo)\cdot z}{|z|^{n+1}}\,dz\\
  =&~ c(n,1/2) \int_{\delta<|z|<1}\dfrac{1}{|z|^{n+1}}\Big[\psi(\xo+z,\yo+z) - \psi(\xo,\yo) - \varepsilon (\xo+\yo)\cdot z \\&\hspace{6cm}+ \phi(\xo +z,\yo+z) - \phi(\xo,\yo)\Big]\,dz\\
  \le &~ c(n,1/2)  \int_{\delta<|z|<1}\frac{ \phi(\xo +z,\yo+z) - \phi(\xo,\yo)- \varepsilon (\xo+\yo)\cdot z}{|z|^{n+1}}\,dz\\
  = &~ c(n,1/2)   \int_{\delta<|z|<1}\frac{\varepsilon |z|^2}{|z|^{n+1}}\,dz = \varepsilon \omega(n)(1-\delta)c(n,1/2).
\end{align*}
Moreover, note that
\begin{align*}
 \laplasK_2 u(\yo)- \laplasK_2 u(\xo)
=& c(n,1/2) \int_{|z|>1} \dfrac{u(\xo+z) -u(\yo+z) - (u(\xo) - u(\yo))}{|z|^{n+1}} \ dz.
\end{align*} It is also fairly straightforward to conclude that 
\begin{align*} 
\lim_{\varepsilon \rightarrow 0}M_{\varepsilon} = \displaystyle \sup_{x,y \in \Rn} \{u(x) -u(y) -
\dfrac{\gamma}{2}|x-y|^2 \}:= M
\end{align*} Therefore for every $\gamma \in (0,\infty )$, there is $\varepsilon_\gamma > 0$ such that 
\begin{align}\label{esti:3}
 0< \displaystyle \sup_{x,y \in \Rn} \{u(x) -u(y) -\dfrac{\gamma}{2}|x-y|^2 \} - M_{\varepsilon} < \frac{1}{\gamma} ,
\end{align} for all $\varepsilon < \varepsilon_\gamma$. We now use the fact that $ u(\xo) -u(\yo) -\frac{\gamma}{2} |\xo-\yo|^2 \geq  M_{\varepsilon}$ and conclude
\begin{align*}
& u(\xo+z)-u(\yo+z) -(u(\xo)-u(\yo))\\ \leq & u(\xo+z) -u(\yo+z)-\frac{\gamma}{2}|\xo-\yo|^2-M_{\varepsilon} \leq \frac{1}{\gamma},
\end{align*} whenever $\varepsilon < \varepsilon_\gamma$. 

Therefore,  for every $\delta \in (0,1)$ and $\gamma\in (0,\infty)$, we have

\begin{align}
\label{eq:nonlocal-term} \laplask \phi_2(\yo) -\laplask \phi_1(\xo) - \laplasK u(\xo)+\laplasK u(\yo) 
\leq C \left(\varepsilon + \delta\gamma+\delta + \frac{1}{\gamma} \right)
\end{align} whenever $\varepsilon$ is small enough.  We now combine \eqref{lip.cont0},\eqref{lip.cont1} and \eqref{eq:nonlocal-term}, and substitute $\varepsilon =\delta $ to obtain

\begin{align}
\label{eq:nonlocal-estimate}\lambda M_{\varepsilon}  \leq K|x_0-y_0| - \dfrac{\gamma}{2}(\lambda - 2K_1) |x_0 -y_0|^2 + B \varepsilon(\gamma + 1+ |x_0|^2 + |y_0|^2 ) + C \frac{1}{\gamma}.
\end{align} Now we assume $\lambda_0 > 2K_1$ and \eqref{eq:nonlocal-estimate} implies

\begin{align}
\notag \lambda M_{\varepsilon}  \leq & \sup_{r\in(0,\infty)} \Big[K r - \dfrac{\gamma}{2}(\lambda - 2K_1) r^2\Big] + B \varepsilon(\gamma + 1+ |x_0|^2 + |y_0|^2 ) + C \frac{1}{\gamma}\\
      = & \frac{K^2}{2\gamma(\lambda-2K_1)}+B \varepsilon(\gamma + 1+ |x_0|^2 + |y_0|^2 ) + C \frac{1}{\gamma} \label{eq:nonlocal-estimate-1}
\end{align}

Now letting $\varepsilon\rightarrow 0$ in \eqref{eq:nonlocal-estimate-1} and noticing that $\lim_{\varepsilon \rightarrow 0}\varepsilon(|x_0|^2+|y_0|^2) =0$, we obtain 
\begin{align}\label{lip:one}
 \lambda \displaystyle \sup_{x,y\in \Rn} \{ u(x)-u(y) -\frac{\gamma}{2} |x-y|^2 \} 
\leq \frac{K^2}{2\gamma (\lambda - 2K_1)} +C \frac{1}{\gamma}
\end{align} for all $\gamma >0$. Therefore, for all $x, y\in \R^n$,

$$ |u(x) -u(y)| \leq \displaystyle \inf_{\gamma>0} \left( \frac{K^2}{2\gamma \lambda (\lambda - 2K_1)} 
+  \frac{C}{\lambda\gamma}  + \frac{\gamma}{2}|x-y|^2 \right)= \tilde{K} |x-y| ,$$
which completes the proof.

\end{proof}

\medskip
Received xxxx 20xx; revised xxxx 20xx.
\medskip

\end{document}